\documentclass[a4paper,12pt]{article}
\usepackage[english]{babel}
\usepackage[T1]{fontenc}
\usepackage[latin1]{inputenc}
\usepackage{graphics, enumerate}
\usepackage{amsmath,amsfonts,amstext,amssymb,color,epsfig,bbm}

\numberwithin{equation}{section}

\newtheorem{stheorem}{Theorem}[section]
\newtheorem{definition}[stheorem]{Definition}
\newtheorem{proposition}[stheorem]{Proposition}
\newtheorem{lemma}[stheorem]{Lemma}
\newtheorem{corolary}[stheorem]{Corollary}
\newenvironment{proof}{{\it \underline{Proof}}\,:\ }{\qed $\\$}

\newenvironment{remark}{{\ \ \ \it Remark\,: }}{ \\}

\def\phi{\varphi}
\def\D{\Delta}

\def\tpar{/ \! /}
\def\norm{| \! | \! |}
\def\R{\mathbb{R}}
\def\E{\mathbb{E}}

\def\1{\mathbbm{1}}

\def\qed{\quad\hbox{$\vcenter{\vbox{
   \hrule height 0.4pt\hbox{\vrule width 0.4pt height 6pt
    \kern5pt\vrule width 0.4pt}\hrule height 0.4pt}}$}}
\def\mathpal#1{\mathop{\mathchoice{\text{\rm #1}}%
   {\text{\rm #1}}{\text{\rm #1}}%
   {\text{\rm #1}}}\nolimits}
\newcommand\Ric{\mathpal{Ric}}
\newcommand\Id{\mathpal{Id}}
\newcommand\ev{\mathpal{ev}}
\newcommand\Tr{\mathpal{Tr}}
\title{Brownian motion with respect to time-changing Riemannian metrics, applications to Ricci flow}

\author{K.A. Coulibaly}

\begin{document}
\date{}
\maketitle

\selectlanguage{english}


\begin{abstract}
We generalize Brownian motion on a Riemannian manifold to the case of a family of metrics which depends on time. Such questions are natural for equations like the heat equation with respect to time dependent Laplacians (inhomogeneous diffusions). In this paper we are in particular interested in the Ricci flow which provides an intrinsic family of time dependent metrics. We give a notion of parallel transport along this Brownian motion, and establish a generalization of the Dohrn-Guerra or damped parallel transport, Bismut integration by part formulas, and gradient estimate formulas. One of our main results is a characterization of the Ricci flow in terms of the damped parallel transport. At the end of the paper we give a canonical  definition of the damped parallel transport in terms of stochastic flows, and derive an intrinsic martingale which may provide information about singularities of the flow.
\end{abstract}

\section{ $g(t)$-Brownian motion }

Let $M$ be a compact connected $n$-dimensional manifold which carries a family of time-dependent Riemannian metrics $g(t)$. In this section we will give a generalization of the well known Brownian motion on $M$ which will depend on the family of metrics.  In other words, it will depend on the deformation of the manifold. Such family of metrics will naturally come from geometric flows like mean curvature flow or Ricci flow. The compactness assumption for the manifold is not essential.
Let $\nabla^{t}$ be the Levi-Civita connection associated to the metric $g(t)$, $\Delta_{t} $ the associated Laplace-Beltrami operator. 
Let also $(\Omega, (\mathcal{F}_{t})_{t \ge 0},\mathcal{F},\mathbb{P})$ be a complete probability space endowed with a filtration $(\mathcal{F}_{t})_{t \ge 0} $ satisfying ordinary assumptions like right continuity and $W$ be a $\R^{n}$-valued Brownian motion for this probability space. 

\begin{definition}
Let us take $(\Omega, (\mathcal{F}_{t})_{t \ge 0},\mathcal{F},\mathbb{P})$ and a $C^{1,2}$-family $g(t)_{t\in [0,T[}$ of metrics over $M$. An $M$-valued process $X(x)$ defined on $ \Omega \times [0,T[$ is called a $g(t)$-Brownian motion in $M$ started at $x\in M$ if $X(x)$ is continuous, adapted, and if for every smooth function $f$,
  $$ f(X_{s}(x)) -f(x) - \frac12 \int_{0}^{s} \D_{t}f (X_{t}(x)) \,dt  $$  
is a local martingale.  
\end{definition}

We shall prove existence of this inhomogeneous diffusion and give a notion of parallel transport along this process.

Let $(e_{i})_{i \in [1..d]}$ be an orthonormal basis of $\R^{n}$, $\mathcal{F}(M)$ the frame bundle over $M$, $\pi$ the projection to $M$. For any $ u \in \mathcal{F}(M)$, let $L_{i}(t,u) = h^{t}(u e_i)$ be the $ \nabla^{t} $ horizontal lift of $ue_i$ and $L_{i}(t)$ the associated vector field. Further let $ V_{\alpha, \beta}$ be the canonical basis of vertical vector fields over $\mathcal{F}(M)$ defined by $V_{\alpha, \beta} (u) = Dl_{u}(E_{\alpha, \beta})$ where $E_{\alpha, \beta}$ is the canonical basis of $\mathcal{M}_{n}(\R)$ and where
$$l_{u} : GL_{n}(\R) \to  \mathcal{F}(M) $$ is the left multiplication. Finally let $ (\mathcal{O}(M),g(t))$ be the $g(t)$ orthonormal frame bundle.

\begin{proposition}\label{EDS_1}
Assume that   $g(t)_{t\in [0,T[}$ is a $C^{1,2}(t,x)$-family of metrics over $M$, and 
 $$\begin{array}{rcl}
A : {[0,T[} \times {\cal{F}}(M) & \rightarrow & {\cal M}_n (\R) \\
      (t,U)    & \mapsto & (A_{\alpha , \beta} (t,U))_{\alpha , \beta}
\end{array}$$
is locally Lipschitz in $ U$ unformly in all compact of $t$. Consider the Stratonovich differential equation in $\mathcal{F}(M)$:
\begin{equation} \label{eds-U} \left\{ \begin{array}{l} 
  *dU_t = \sum_{i=1}^n L_i(t,U_t)*dW^i +\sum_{\alpha , \beta}A_{\alpha , \beta}(t,U_t)  V_{\alpha , \beta}(U_t) \,dt \\
  U_0 \in \mathcal{F}(M) \text{ such that } U_0 \in (\mathcal{O}(M),g(0)) .
\end{array}
\right .
\end{equation}\\
Then there is a unique symmetric choice for $A$ such that  $U_{t} \in (\mathcal{O}(M),g(t)) .$
Moreover:
 $$  A(t,U) =  - \frac12 \partial_{1} G(t, U),$$ 
where  $(\partial_1 G(t, U))_{i,j} = \left\langle U e_i , U e_j \right\rangle_{\partial_{t} g(t)}$.
\end{proposition}

\begin{proof}
Let us begin with curves. Let $I$ be a real interval, $\pi:TM \rightarrow M$ the projection, $V$ and $C$ in  $C^1(I , TM)$, two curves  such that $$x(t) := \pi(V(t))=\pi(C(t))     \text{, for all } t \in I $$
We want to compute:
$$ \frac{d}{dt}_{\mid_{t=0}}  \Big(  \left\langle V(t) , C(t) \right\rangle_{g(t,x(t))} \Big) $$

We write $ \partial_1 g(t,x)$ for $  \partial_s g(s,x)$ evaluated at $t$. Let us express the metric $g(t)$ in a coordinate system;
without loss of generality we can differentiate at time $0$. Let $(x^{1},...,x^{n})$ be a coordinate system at the point $x(0)$, in which we have:
$$V(t) = v^{i}(t) \partial_{x^{i}} $$
$$ C(t) = c^{i}(t) \partial_{x^{i}} $$
$$ g(t,x(t)) = g_{i,j}(t,x(t)) dx^{i}\otimes dx^{j} $$

In these local coordinates we get:
$$\begin{array}{lcl}
\displaystyle\frac{d}{dt}_{\mid_{t=0}}    \left\langle V(t) , C(t) \right\rangle_{g(t,x(t))} &=& 
\displaystyle\frac{d}{dt}_{\mid_{t=0}} g_{i,j}(t,x(t)) v^{i}(t)c^{j}(t)  \\
&=& (\partial_1 g_{i,j}(0,x)v^{i}(0)c^{j}(0) + \displaystyle\frac{d}{dt}_{\mid_{t=0}} (g_{i,j}(0,x(t)) v^{i}(t)c^{j}(t)) \\
&=& \partial_1 g_{i,j}(0,x)v^{i}(0)c^{j}(0) + \left\langle \nabla^{0}_{\dot x(0)}V(0) , C(0) \right\rangle_{g(0,x(0))} \\
& & + \left\langle V(0) , \nabla^{0}_{\dot x(0)}C(0) \right\rangle_{g(0,x(0))}\\
&=& \left\langle V(0),C(0) \right\rangle_{\partial_1 g(0,x(0)}  + \left\langle \nabla^{0}_{\dot x(0)}V(0) , C(0) \right\rangle_{g(0,x(0))}\\
& & + \left\langle V(0) , \nabla^{0}_{\dot x(0)}C(0) \right\rangle_{g(0,x(0))}.
\end{array}$$
 
In order to compute the $g(t)$ norm of a tangent valued process we will use what Malliavin calls ``the transfer principle'', as explained in \cite{Emery-transfer},\cite{Emery}.

Recall the equivalence between a given connection on a manifold $M$ and a splitting on $TTM$, i.e. $TTM = H^{\nabla}TTM \oplus VTTM $ \cite{Kob-Nom}. We have a bijection:
 $$\begin{array}{rcl}
  \mathcal{V}_{v} :T_{\pi(v)}M &\longrightarrow& V_{v}TTM \\
  u &\longmapsto& \frac{d}{dt} (v+tu)|_{t=0}.
\end{array}$$
For $  X,Y \in \Gamma(TM)$ we have:
$$ \nabla_{X}Y (x) =  \mathcal{V}_{X(x)}^{-1}((dY(x)(X(x)))^{v}),$$
where $ (.)^{v}$ is the projection of a vector in  $TTM$ onto the vertical subspace $VTTM $ parallely to $H^{\nabla}TTM $.

For a $ T(M)$-valued process $T_{t}$, we define:
 \begin{equation}\label{cov_mart}
    D^{S,t}T_t  = ({\cal V} _{T_t })^{-1}((*dT_t )^{v,t}),
 \end{equation}
where $ (.)^{v,t}$ is defined as before but for the connection $\nabla^{t} $. The above generalization makes sense for a tangent valued process coming from a Stratonovich equation like $U_t e_i $, where $U_{t}$ is a solution of the Stratonovich differential equation (\ref{eds-U}).
 
For the solution $U_{t}$ of (\ref{eds-U}) we get
 \begin{eqnarray}\label{dst-equation}
  d \left(  \left\langle U_t e_i , U_t e_j  \right\rangle_{g(t,\pi (U_t))} \right) &=&\left\langle U_t e_i , U_t e_j \right\rangle_{ \partial_1 g(t,\pi (U_t)))}\,dt\\ 
   & & + \left\langle D^{S,t}U_t e_i , U_t e_j \right\rangle_{g(t,\pi (U_t))} + \left\langle   U_t e_i ,D^{S,t} U_t e_j \right\rangle_{g(t,\pi (U_t))}
 \end{eqnarray}
We would like to find a symmetric  $A$ such that the left hand side of the above equation vanishes for all time (i.e. $U_{t} \in (\mathcal{O}(M),g(t))$).
Denote by
$ \ev_{e_i} : \mathcal{F}(M) \to TM $ the ordinary evaluation, and  
$d\ev_{e_i} : T\mathcal{F}(M) \to TTM$ its differential.\\
 It is easy to see that $d\ev_{e_i}$ sends $VT\mathcal{F}(M)$ to $VTTM$ and sends $H^{\nabla^h}T\mathcal{F}(M)$ to $H^{\nabla} TTM$.
 We obtain:
\begin{eqnarray} 
D^{S,t}U_t e_i 
                &=& \sum_{\alpha = 1}^n A_{\alpha ,i}(t,U_t)U_te_{\alpha}\,dt. \label{Ddef}
\end{eqnarray} 
For simplicity, we take for notation:
 $(\partial_1 G(t, U))_{i,j} = \left\langle U e_i , U e_j \right\rangle_{\partial_{t} g(t)}$ and  
$$(G(t, U))_{i,j} = \left\langle U e_i , U e_j \right\rangle_{g(t)}.$$
 It is now easy to find the condition for $A$:
\begin{equation}
 (G(t,U_t)A(t,U_t))_{j,i} + (G(t,U_t)A(t,U_t))_{i,j}=-(\partial_1 G(t, U_t))_{i,j} \label{eq01}
\end{equation}
Given orthogonality $G(t,U_{t})=\Id $ and so by (\ref{eq01}) $A$ differs from $ - \frac12 \partial_1 G $ by skew symmetric matrice, therefore will be equal to it if we demand symmetry.
Conversely if $ A =  - \frac12 \partial_1 G $ then by (\ref{dst-equation}) and equation (\ref{cov_mart}) we see $G(t,U_{t})=\Id $.

%
\end{proof}

\begin{remark}
The SDE in proposition \ref{EDS_1} does not explode because on any compact time interval all coefficients and their derivatives up to order 2 in space and order~1 in time are bounded.
\end{remark}

\begin{remark}
The condition of symmetry is linked to a good definition of parallel transport with moving metrics in some sense.
\end{remark}

To see where the condition of symmetry comes from we may observe what happens in the constant metric case. It is easy to see that the usual definition of parallel transport along a semi-martingale which depends on the vanishing of the Stratonovich integral of connection form, is equivalent to isometry and the symmetry condition for the drift in the following SDE in $\mathcal{F}(M)$:
$$\left\{ \begin{array}{l}
  d\tilde{U}_t = \sum_{i=1}^d L_i(\tilde{U}_t)*dW^i +A( \tilde{U}_t)_{\alpha , \beta}   V_{\alpha , \beta}(\tilde{U}_t) \,dt \\
  \tilde{U}_0 \in (\mathcal{O}(M),g)\\
   \quad \tilde{U}_t \in (\mathcal{O}(M),g) \quad \text{(isometry)}\\
  \quad A( . , . )_{\alpha , \beta} \in S(n) \quad \text {(vertical evolution)}.
\end{array}
\right .$$
\begin{remark}
 Isometry of $U_{t}$ forces $A$ to be skeew symmetric by (\ref{eq01}), the symmetry of $A$ give $A=0$. We get the usual stochastic differential equation of the parallel transport in constante metric case.
\end{remark}

 The next proposition is a direct adaptation of a proposition  in \cite{Hsu}, page 42; hence the 
proof is omitted.
 
 \begin{proposition}
 Let $\alpha \in \Gamma(T^{*}M)$ and $F_{\alpha} : {\cal F}(M) \to \ \R^{d}$,   
$F_{\alpha}^{i}(u) = \alpha_{\pi(u)}(ue_i)$ its scalarization.
Then, for all $A \in \Gamma(TM)$, 
$$(\nabla_{A} \alpha)_{\pi (u)}(ue_i) = h(A_{\pi(u)}) F_{\alpha}^{i} .$$ 
 \end{proposition}

Consequently, for all $u \in {\cal F}(M) $,
$$ (\nabla_{A}^{g(t)} df)_{\pi (u)}(u e_i) = h^{g(t)}(A_{\pi(u)}) F_{df}^{i}   $$
and for $f \in C^{\infty}(M)$,
\begin{eqnarray*}
   L_{i}(t)(f \circ \pi)(u) &=& d(f \circ \pi) L_{i}(t,u)\\
   &=&  F_{df}^{i}(u).
\end{eqnarray*}
Hence we have the formula: 
\begin{eqnarray*}
L_{i}(t)L_{j}(t)(f \circ \pi)(u)
&=& h^{g(t)}(u e_i) F_{df}^{j} \\
&=& (\nabla_{u e_i}^{g(t)} df)(u e_j)  \\
&=& \nabla^{g(t)}df(u e_i , u e_j).
\end{eqnarray*}

\begin{proposition}\label{mbg(t)}
Take $x \in M$ and the SDE in $\mathcal{F}(M)$:
\begin{equation}
\left\{ \begin{array}{l} \label{eq-relevemant}
  *dU_t = \sum_{i=1}^n L_i(t,U_t)*dW^i -\frac12 \partial_{1}G(t,U_t)_{\alpha , \beta}  V_{\alpha , \beta}(U_t) \,dt \\
  U_0 \in \mathcal{F}(M) \text{ such that } U_0 \in (\mathcal{O}_{x}(M),g(0)) .
\end{array}
\right .\end{equation}
 Then $X_{t}(x)= \pi(U_{t}) $ is a $g(t)$-Brownian motion, which we note $g(t)$-BM(x).
\end{proposition}

\begin{proof}
For $f \in C^{\infty}(M)$,
$$
\begin{array}{lcl}
  d(f \circ \pi \circ U_{t}) &=& \sum_{i=1}^{n} L_{i}(t)(f \circ \pi)(U_{t})*dW^{i} \\
   &=& \sum_{i=1}^{n} L_{i}(t)(f \circ \pi)(U_{t})dW^{i} + \frac12 \sum_{i,j=1}^{n} L_{i}(t)L_{j}(t)(f \circ \pi)dW^{i}dW^{j} \\
  &\overset{d \cal M }{\equiv} & \frac12 \sum_{i=1}^{n}\nabla^{g(t)}df(U_t e_i , U_t e_i)\,dt\\
  &\overset{d \cal M }{\equiv} & \frac12 \D_{t}f(\pi \circ U_{t}) \,dt.
\end{array}$$
The last equality comes from the fact that $U_{t} \in (\mathcal{O}(M),g(t)) $.
\end{proof}

\begin{remark}
Recall that in the compact case the lifetime of equation (\ref{eq-relevemant}) is deterministic and the same as the lifetime of the metrics family. 
\end{remark}

Let $U_{t}$ be the solution of (\ref{eq-relevemant}). 
We will write $\tpar_{0,t} =U_{t} \circ U_{0}^{-1} $ the $g(t)$ parallel transport over a $g(t)$-Brownian motion ( we call it parallel transport because it is a natural extention of the usual parallel transport in the constante metric case). As usual it is an isometry: 
$$\tpar_{0,t} : (T_{X_{0}}M, g(0)) \to  (T_{X_{t}}M, g(t)). $$
We also get a development formula. Take an orthonormal basis $(v_{1},...,v_{n})$ of $(T_{X_{0}}M, g(0))$, and $X_{t}(x)$ a $g(t)$-Brownian motion of proposition \ref{mbg(t)}; then
$$ *dX_{t}(x) = \tpar_{0,t}v_{i} *dW^{i}_{t}. $$
For $f \in C^{2}(M)$ we get the It\^{o} formula:
\begin{equation} \label{g(t)-ito}
 df(X_{t}(x)) = \langle \nabla^{t} f ,\tpar_{0,t}v_{i} \rangle_{t}dW^{i} + \frac12 \D_{t}(f)(X_{t}(x)) \,dt .
\end{equation}

We will now give examples of $g(t)$-Brownian motion.
Let $(S^{n}, g(0))$ be a sphere and the solution of the Ricci flow: $\frac{\partial}{\partial_{t}}g(t) = - 2 \Ric_{t} $ that is $g(t)=(1-2(n-1)t)g(0) $ with explosion time $T_{c} = \frac{1}{2(n-1)} $. 
We will use the fact that all metrics are conformal to the initial metric to express the $g(t)$-Brownian motion in terms of the $g(0)$-Brownian motion.
Let $ f \in C^{2}(S^{n})$, $X_{t}(x)$ be a $g(t)$-Brownian motion starting at $x \in  S^{n} $, $B_{t}$ some real-valued Brownian motion, and $\mathbb{B}_{t}(x)$ a $S^{n}$ valued $g(0)$-Brownian motion. Then:
$$ df(X_{t}(x)) =  \parallel \nabla^{t} f (X_{t}(x))\parallel_{g(t)} dB_{t} + \frac12 \left(\frac{1}{1-2(n-1)t}\right) \D_{0}f(X_{t}(x)) \,dt.$$
We have:
$$\parallel \nabla^{t} f \parallel_{g(t)}^{2} = \frac{1}{1-2(n-1)t} \parallel \nabla^{0} f \parallel_{0}^{2}. $$
Let 
 $$\tau(t) = \int_{0}^{t} \frac{1}{1-2(n-1)s} \,ds ,$$ 
  then $$\tau(t)= \frac{\ln(1-2(n-1)t)}{-2(n-1)}  ,\quad \tau^{-1}(t)= \frac{e^{-2(n-1)t}-1}{-2(n-1)} .$$
We have the equality in law:
$$(X_{.}(x)) \overset{\mathcal{L}}{=} (\mathbb{B}_{\tau(.)}(x)) . $$

We have a similar result for the hyperbolic case:
Let $ (H^{n}(-1),g(0)) $ be the hyperbolic space with constant curvature $-1$. Then $g(t)=(1+2(n-1)t)g(0)$ is the solution of the Ricci flow. Let $X_{t}(x)$ be a $g(t)$-Brownian motion starting at $x \in  S^{n} $,and $\mathbb{B}_{t}(x)$ an $H^{n}$-valued $g(0)$-Brownian motion. Then:
$$\tau(t) = \int_{0}^{t} \frac{1}{1+2(n-1)s} \,ds, $$ 
and in law:
$$ (X_{.}(x) )\overset{\mathcal{L}}{=}( \mathbb{B}_{\tau(.)}(x)) .$$

Let us look at what happens for some limit of the Ricci flow, the so called Hamilton cigar manifold (\cite{chowknopf}).
Let on $ \R^{2}$, $g(0,x)= \frac{1}{1+\parallel x \parallel^{2}}\, g_{\rm can} $ be the Hamilton cigar , where $\parallel . \parallel $ is the Euclidean norm. Then the solution to the Ricci flow is given by $ g(t,x) =\frac{1+\parallel x\parallel^{2}}{e^{4t}+\parallel x\parallel^{2}}g(0,x)  $. Let $ f \in C^{2}(\R^{2})$, $X_{t}(x)$ be a $g(t)$-Brownian motion starting at $x \in \R^{2}  $, $B_{t}$ a real-valued Brownian motion, and $\mathbb{B}_{t}(x)$ some $\R^{2}$ valued $g(0)$-Brownian motion. Then:
$$ df(X_{t}(x)) =  \parallel \nabla^{t} f (X_{t}(x))\parallel_{g(t)} dB_{t} + \frac12  \frac{e^{4t}+\parallel X_{t}(x) \parallel^{2}}{1+\parallel X_{t}(x)\parallel^{2}} \D_{0}f(X_{t}(x)) \,dt.$$
We have:
$$ \nabla^{t} f(x)  = \frac{e^{4t} + \parallel x \parallel^{2}}{1 + \parallel x \parallel^{2}} \nabla^{0} f (x) ,$$ 
$$ \parallel \nabla^{t} f(x) \parallel_{t}^{2} = \frac{e^{4t} + \parallel x \parallel^{2}}{1 + \parallel x \parallel^{2}} \parallel \nabla^{0} f (x)\parallel_{0}^{2}, $$
$$\D_{t}f= \frac{e^{4t} + \parallel x \parallel^{2}}{1 + \parallel x \parallel^{2}} \D_{0}f. $$
We set:
$$\tau(t) = \int_{0}^{t} \frac{e^{4s} + \parallel X_{s}(x) \parallel^{2}}{ 1 + \parallel X_{s}(x) \parallel^{2}} \,ds .$$ 
Then in law:
$$(X_{.}(x)) \overset{\mathcal{L}}{=} ( \mathbb{B}_{\tau(.)}(x) ) $$

\begin{remark}
If $X_{t}(x)$ is a $g(t)$-Brownian motion associated to a Ricci flow started at $g(0) $ then $X_{{t}/{c}}(x)$ is a $cg({t}/{c})$-Brownian motion associated to a Ricci flow started at $cg(0) $ so it is compatible with the blow up.
\end{remark}

\section{Local expression, evolution equation for the density, conjugate heat equation}

We begin this section by expressing a $g(t)$-Brownian motion in local coordinates.

\begin{proposition} \label{eq-local}
 Let $x\in M$, $(x^{1},...,x^{n})$ be local coordinates around $x$, and $X_{t}(x)$ a $g(t)$-Brownian motion.
Before the exit time of the domain of coordinates, we have: 
$$ dX^{i}_{t}(x) = \sqrt{g^{-1}(t,X_{t}(x))_{i,j}} dB^{j} - \frac12 g^{k,l} \Gamma^{i}_{kl}(t,X_{t}(x)) \,dt $$
where we denote $ \sqrt{g^{-1}(t,X_{t}(x))_{i,j}} $ the unique positive square root of the inverse to the matrix $(g(t, \partial_{x^{i}},\partial_{x^{j}}))_{i,j} $; here $\Gamma^{i}_{kl}(t,X_{t}(x))$ are the Christoffel symbols associated to $\nabla^{g(t)} $, and $B^{i}$ are  $ n$ independent Brownian motion. 
\end{proposition}
\begin{proof}
From the It\^o equation \ref{g(t)-ito}, we get:
$$ dX^{i}_{t}(x) = \langle \nabla^{t} x^{i}  ,  \tpar_{0,t} v_{l}   \rangle_{g(t)}  dW^{l} + \frac12 \D_{t} x^{i} (X_{t}(x)) dt,$$
where $(v_{1},...,v_{n})$ is a $g(0)$-orthogonal basis of $T_{x}M$. By the usual expression of the Laplacian in coordinates:
$$ \D_{t} x^{i} (X_{t}(x)) = -g^{l,k} \Gamma_{kl}^{i}(t, X_{t}(x)), $$ 
and the gradient expression of the coordinates functions: 
$$\nabla^{t}x_{i} = g(t)^{i,j} \frac{\partial }{\partial x_{j}}, $$
we have:
\begin{eqnarray*}
 dX^{i}_{t}(x) &=& g(t)^{i,j} \langle \frac{\partial }{\partial x_{j}} ,  \tpar_{0,t} v_{l}   \rangle_{g(t)}  dW^{l} - \frac12 g^{l,k} \Gamma_{kl}^{i}(t, X_{t}(x)) \,dt \\
 &=& \sum_{m} \sqrt{g(t)^{i,m}} \langle \sqrt{g(t)^{m,j}} \frac{\partial }{\partial x_{j}} ,  \tpar_{0,t} v_{l}   \rangle_{g(t)}  dW^{l} - \frac12 g^{l,k} \Gamma_{kl}^{i}(t, X_{t}(x)) \,dt \\
&=& \sqrt{g(t)^{i,m}} dB^{m}   - \frac12 g^{l,k} \Gamma_{kl}^{i}(t, X_{t}(x)) \,dt  \, ,\\
\end{eqnarray*}
where $ dB^{m} = \langle \sqrt{g(t)^{m,j}} \frac{\partial }{\partial x_{j}} ,  \tpar_{0,t} v_{l}   \rangle_{g(t)}  dW^{l} $.
By the isometry property of the parallel transport and L\'evy's theorem $B=(B^{1},...,B^{n})$ is a Brownian motion in $\R^{n} $.

\end{proof}

\begin{remark}
 The above equation is similar to the equation in the fixed metric case. 
\end{remark}

Now we shall study the evolution equation for the density of the law of the $g(t)$-Brownian motion. Let $X_{t}(x)$ be a $g(t)$-BM$(x)$, and $d\mu_{t}$ the Lebesgue measure over $(M,g(t))$. 
$X_{t}(x)$ is a diffusion with generator $\D_{t}$, we have the smoothness of the density (e.g. \cite{ Stroock}).
Let  $ h^{x}(t,y) \in C^{\infty}(]0,T[ \times M )$ such that:
$$\left\{\begin{array}{lcl}
 X_{t}(x) &\overset{\mathcal{L}}{=}& h^{x}(t,y) d\mu_{t}(y), t>0 \\
 X_{0}(x) &\overset{\mathcal{L}}{=}& \delta_{x}.
 \end{array}\right .$$
By the continuity of $X_{t}(x)$ and the dominated convergence theorem we get the convergence in law:
 $$  \overset{\mathcal{L}}{\lim_{t \to 0}} \quad   X_{t}(x) = \delta_{x} .$$
We write in a local chart the expression of $d\mu_{t}$ in terms of $d\mu_{0}$, i.e.,
$$d\mu_{t} = \frac{\sqrt{\det(g_{i,j}(t))}} {\sqrt{\det(g_{i,j}(0))}}  \sqrt{\det(g_{i,j}(0))} |dx^{1} \wedge dx^{2} \wedge ...  \wedge dx^{n}|$$
and we note:
$$ \mu_{t}(dy)= \psi(t,y) \mu_{0}(dy) .$$
\begin{proposition}
 $$\left\{\begin{array}{l}
  \displaystyle{ 
  \frac{d}{dt}(h^{x}(t,y)) +  h^{x}(t,y) \Tr \bigg{(} \frac12 (g^{-1}(t,y)) \frac{d}{dt} g(t,y) \bigg{)} = \frac12 \D_{g(t)}h^{x}(t,y)}\\
  \displaystyle{ \overset{\mathcal{L}}{\lim_{t \to 0}} \quad h^{x}(t,y) d\mu_{t} = \delta_{x}} .
   \end{array} \right . $$
\end{proposition}
\begin{proof}
 For $f \in C^{\infty}(M) $, $t>0$, by definition of $X_{t}(x)$ we have:
$$ \begin{array}{l}
 \E [f(X_{t}(x))] - f(x) = \frac12  \E \bigg{[} \int_{0}^{t} \D_{g(s)} f (X_{s}(x)) \,ds  \bigg{]} \\ 
 \frac{d}{dt} \E [f(X_{t}(x))] = \frac12 \E[\D_{g(t)} f (X_{t}(x)) ], 
\end{array}$$
i.e.:
$$ \begin{array}{lcl}
  \frac{d}{dt} \int_{M} h^{x}(t,y)f(y) \mu_{t}(dy) &=& \frac12 \int_{M} \D_{g(t)} f (y) h^{x}(t,y) \mu_{t}(dy)\\
                &=& \frac12 \int_{M}  f (y) \D_{g(t)}h^{x}(t,y) \mu_{t}(dy). 
   \end{array} $$
The last equality comes from Green's theorem and the compactness of the manifold.
By changing $ \mu_{t}(dy)= \psi(t,y) \mu_{0}(dy) $ in the left hand side, we have:
 $$ \int_{M} f(y) \frac{d}{dt}(h^{x}(t,y)\psi(t,y))\mu_{0}(dy) =  \frac12 \int_{M}  f (y) (\D_{g(t)}h^{x}(t,y)) \psi(t,y)) \mu_{0}(dy) $$
so:
\begin{equation}
 \label{dens} \frac{d}{dt}(h^{x}(t,y)\psi(t,y)) = \frac12 (\D_{g(t)}h^{x}(t,y)) \psi(t,y)) 
 \end{equation}
We also have by determinant differentiation:
$$\begin{array} {lcl}
\displaystyle\frac{d}{dt} \psi(t,y) &=& \displaystyle\frac{1}{2\sqrt{\det(g_{i,j}(0))}} \frac{1}{\sqrt{\det(g_{i,j}(t))}}
\det(g_{i,j}(t)) \, \Tr \bigg{(} g^{-1}(t,y) \frac{d}{dt} g(t,y) \bigg{)} \\
 &=& \displaystyle\frac12 \psi(t,y)\, tr \bigg{(}  g^{-1}(t,y) \frac{d}{dt} g(t,y) \bigg{)} .
\end{array} $$

The part $\Tr \bigg{(} \frac12\, g^{-1}(t,y) \frac{d}{dt} g(t,y) \bigg{)}$  is intrinsic, it does not depend on the choice of the chart. Hence (\ref{dens}) gives the following inhomogeneous reaction-diffusion equation:
$$\frac{d}{dt}(h^{x}(t,y)) +   h^{x}(t,y)tr \bigg{(} \frac12 g^{-1}(t,y) \frac{d}{dt} g(t,y) \bigg{)} = \frac12 \D_{g(t)}h^{x}(t,y). $$
\end{proof}

We will give as example the evolution equation of the density in the case where the family of metrics comes from the forward (and resp. backward) Ricci flow. From now Ricci flow will mean (probabilistic convention):
\begin{equation}\label{forwardRicci}
 \begin{array}{l} 
 \frac{d}{dt} g_{i,j} = -\Ric_{i,j}.
 \end{array}
\end{equation}
(respectively)
\begin{equation}\label{backwardRicci}
 \begin{array}{l} 
 \frac{d}{dt} g_{i,j} = \Ric_{i,j}.
 \end{array} 
\end{equation}

\begin{remark}
 Hamilton in \cite{Ham-three}, and later DeTurck in \cite{Deturck} have shown existence in small times of such flow. In this section we don't care about the real existence time.
\end{remark}

For $x \in M$, we will denote by $S(t,x)$ the scalar curvature at the point $x$ for the metric $g(t)$.
\begin{corolary} \label{eq-conjugue}
 For the backward Ricci flow \textup{(\ref{backwardRicci})}, we have:
 $$ \left\{\begin{array}{l}
   \displaystyle{ 
   \frac{d}{dt}(h^{x}(t,y)) + \frac12 h^{x}(t,y) S(t,y)  = \frac12 \D_{g(t)}h^{x}(t,y)}\\
   \displaystyle{ \overset{\mathcal{L}}{\lim_{t \to 0}} \quad h^{x}(t,y) d\mu_{t} = \delta_{x}} .
 \end{array} \right . $$

 For the forward Ricci flow \textup{(\ref{forwardRicci})}, we have:
 $$ \left\{\begin{array}{l}
   \displaystyle{ 
   \frac{d}{dt}(h^{x}(t,y)) - \frac12 h^{x}(t,y) S(t,y)  = \frac12 \D_{g(t)}h^{x}(t,y)}\\
   \displaystyle{ \overset{\mathcal{L}}{\lim_{t \to 0}} \quad h^{x}(t,y) d\mu_{t} = \delta_{x}} .
 \end{array} \right . $$
\end{corolary}

\begin{remark}
 These equations are conservative. This is not the case for the ordinary heat equation with time depending Laplacian i.e. $\D_{g(t)}$. They are conjugate heat equations which are well known in the Ricci flow theory (e.g. \cite{Topping}).
\end{remark}

\section{Damped parallel transport, and Bismut formula for Ricci flow, applications to Ricci flow for surfaces}

In this section, we will be interested in the heat equation under the Ricci flow. The principal fact is that under forward Ricci flow, the damped parallel transport or Dohrn-Guerra transport is the parallel transport defined before.  The deformation of geometry under the Ricci flow compensates the deformation of the parallel transport (i.e. the Ricci term in the usual formula for the damped parallel transport in constant metric case see (\cite {ElvLI}, \cite {thalwang}, \cite {Elwyor})). The isometry property of the damped parallel transport turns out to be an advantage for computations. In particular, for gradient estimate formulas, everything looks like in the case of a Ricci flat manifold with constant metric.
We begin with a general result independent of the fact that the flow is a Ricci flow.
Let $g(t)_{[0,T_{c}[}$ be a $C^{1,2}$ family of metrics, and consider the heat equation:
\begin{equation}\label{heat}
  \left\{\begin{array}{l} 
  \partial_{t}f(t,x) = \frac12 \D_{t} f(t,x) \\
   f(0,x)=f_{0}(x),
   \end{array} \right .
\end{equation}
where $f_{0}$ is a function over $M$. We suppose that the solution of (\ref{heat}) exists until $T_{c}$. For $T < T_{c} $, let $X^{T}_{t}$ be a $g(T-t)$-Brownian motion, $\tpar^T_{0,t}$ the associated parallel transport.

\begin{definition}
We define the damped parallel transport $ \textbf{W}_{0,t}^{T}$  as the solution of:
$$ *d ((\tpar^{T}_{0,t})^{-1} (\textbf{W}_{0,t}^{T})) = -\frac12 (\tpar^{T}_{0,t})^{-1} ( \Ric_{g(T-t)} - \partial_{t}(g(T-t)))^{\# g(T-t)}(\textbf{W}_{0,t}^{T}) \,dt $$
with
$$ \textbf{W}_{0,t}^{T} : T_{x}M \longrightarrow T_{X_{t}^{T}(x)}M , \textbf{W}_{0,0}^{T}=\Id_{T_{x}M}.$$
\end{definition}

\begin{stheorem}\label{damp}
For every solution $f(t, .)$ of \eqref{heat}, and for all $ v \in T_{x}M  $,
$$df(T-t , . )_{X_{t}^{T}(x)} ( \textbf{W}_{0,t}^{T} v)  $$
is a local martingale.
\end{stheorem}

\begin{proof}
 Recall the equation of a parallel transport over the $g(T-t)$-Brownian motion $X^{T}_{t}(x)$:
\begin{equation} \left\{
    \begin{array} {l}
    *dU^{T}_{t}=\sum_{i=1}^d L_i(T-t,U^{T}_t)*dW^i - \frac12 \partial_{t}(g(T-t)) (U^{T}_t e_{\alpha} , U^{T}_t e_{\beta}) V_{\alpha , \beta}(U^{T}_t) \,dt  \\
    U^{T}_{0} \in (\mathcal{O}_{x}(M),g(T)). 
     \end{array} \right . \label{parT}
\end{equation}
For $f\in \mathcal{C}^{\infty}(M) $, its scalarization:
$$\begin{array}{rcl}
\widetilde{df}: \mathcal{F}(M) & \longrightarrow&  \R^{n} \\
             U & \longmapsto & (df(Ue_1), ... , df(Ue_n )  ),
            
\end{array} $$
yields the following formula in $\R^{n} $:
$$\begin{array}{rcl}
 df(T-t , . )_{X_{t}^{T}(x)} ( \textbf{W}_{0,t}^{T}v) &=& \langle \widetilde{df} (T-t , U^{T}_{t}) ,(U^{T}_{t})^{-1} \textbf{W}_{0,t}^{T} v \rangle_{\R^{n}}, \\
\end{array}$$
for every $v \in T_{x}M $.
To recall the notation let:
$$\begin{array}{rcl}
\ev_{e_{i}}: \mathcal{F}(M) & \longrightarrow&  TM \\
             U & \longmapsto & Ue_i
            
\end{array} $$
and recall that $U^{T}_{t} $, solution of (\ref{parT}), is a diffusion associated to the generator 
$$\frac12 \D_{T-t}^{H} - \frac12 \partial_{t}(g(T-t))(\ev_{ei} (.) ,\ev_{ej} (.))  V_{i,j}(.)  $$ 
where $\D_{T-t}^{H} $ is the horizontal Laplacian in $\mathcal{M} $, associated to the metric $g(T-t)$.
In the It\^{o} sense, we get: 
\begin{align*}
 &d(df(T-t , . )_{X_{t}^{T}(x)} ( \textbf{W}_{0,t}^{T})v)= d \langle \widetilde{df} (T-t , U^{T}_{t}) ,(U^{T}_{t})^{-1} \textbf{W}_{0,t}^{T} v \rangle_{\R^{n}} \\
                          &\overset{d \cal M }{\equiv}  \langle  -(\frac{d}{dt} \widetilde{df})(T-t,.)(U^{T}_{t}) dt 
                          +[ \frac12\D_{T-t}^{H}  \widetilde{df}(T-t,.) \\
                         &- \frac12 \partial_{t}(g(T-t))(\ev_{ei} . ,\ev_{ej}. )  V_{i,j}(.) \widetilde{df}(T-t,.)](U^{T}_{t})\,dt ,(U^{T}_{t})^{-1} \textbf{W}_{0,t}^{T}v \rangle_{\R^{n}} \\ 
                         &+  \langle  (\widetilde{df} (T-t , U^{T}_{t})) ,(U^{T}_{0})^{-1} d((\tpar^{T}_{0,t})^{-1} (\textbf{W}_{0,t}^{T}))v  \rangle_{\R^{n}} \\
                          &\overset{d \cal M }{\equiv}  -(\frac{d}{dt} df)(T-t,.)((\textbf{W}_{0,t}^{T}))v) \,dt  +\langle  [ \frac12\D_{T-t}^{H}  \widetilde{df}(T-t,.) \\
&- \frac12 \partial_{t}(g(T-t))(\ev_{ei} . ,\ev_{ej} )  V_{i,j}(.) \widetilde{df}(T-t,.)](U^{T}_{t})\,dt ,(U^{T}_{t})^{-1}\textbf{W}_{0,t}^{T}v \rangle_{\R^{n}} \\ 
&-  \frac12 \langle  (\widetilde{df} (T-t , U^{T}_{t})),(U^{T}_{0})^{-1}(\tpar^{T}_{0,t})^{-1} ( \Ric_{g(T-t)} - \partial_{t}(g(T-t)))^{\# g(T-t)}(\textbf{W}_{0,t}^{T})v \,dt\rangle_{\R^{n}}. 
\end{align*}
We shall make separate computations for each term in the previous equality. 
Using the well known formula (e.g. \cite{Hsu}, page 193)
 $$ \D^{H}  \widetilde{df} =  \widetilde{ \D df} ,$$ 
we first note that:
$$\begin{array}{l}
\langle \frac12\D_{T-t}^{H}  \widetilde{df}(T-t,.) (U^{T}_{t}) , (U^{T}_{t})^{-1}\textbf{W}_{0,t}^{T}v   \,dt \rangle_{\R^{n}}\\
= \frac12 \langle \widetilde{ \D_{T-t} df}(T-t,.) (U^{T}_{t})   ,(U^{T}_{t})^{-1}\textbf{W}_{0,t}^{T}v \rangle_{\R^{n}} \,dt \\
=  \frac12  \D_{T-t} df(T-t,.) (\textbf{W}_{0,t}^{T}v) \,dt ,
\end{array} $$
By definition:
$$\begin{array}{rcl}
{V}_{i,j} \widetilde{df} (u) &=& \frac{d}{dt}|_{t=0} \widetilde{df}(u(\Id + t E_{ij})) \\
                           &=& \frac{d}{dt}|_{t=0} (df (u(\Id + t E_{ij})e_{s}))_{s=1..n} \\
                           &=& (df(u \delta_{i}^{s} e_{j}))_{s=1..n} \\
                            &=& ( 0, ... ,0,df(ue_{j}),0, ... ,0)\quad i\text{-th position}, 
                            \end{array} $$
so that:
 $$\begin{array}{l}
\sum_{ij} \partial_{t}(g(T-t))(\ev_{ei} . ,\ev_{ej} . ) V_{i,j}(.) \widetilde{df}(T-t,.)(U^{T}_{t})\,dt \\
= \sum_{ij} \partial_{t}(g(T-t))(U^{T}_{t}e_{i} ,U^{T}_{t}e_{j} ) df(U^{T}_{t}e_{j}) e_{i} \,dt \\
= (\langle \nabla^{T-t}f(T-t , .) ,\sum_{j}\partial_{t}(g(T-t))(U^{T}_{t}e_{i} ,U^{T}_{t}e_{j} )U^{T}_{t}e_{j} \rangle_{T-t}\,dt)_{i=1..n}\\
=(df(T-t , \partial_{t}(g(T-t))^{\# T-t}(U^{T}_{t}e_{i}))\,dt)_{i=1..n}.
\end{array} $$
Then 
$$\begin{array}{l}
 d(df(T-t , . )_{X_{t}^{T}(x)} ((\textbf{W}_{0,t}^{T})v))    \\                      
                          \overset{d \cal M }{\equiv}  -\frac{d}{dt} df(T-t,.)((\textbf{W}_{0,t}^{T}v ) \,dt  \\
                          - \frac12 \langle (df(T-t , \partial_{t}(g(T-t))^{\# T-t}(U^{T}_{t}e_{i})))_{i=1..n} , (U^{T}_{t})^{-1} \textbf{W}_{0,t}^{T}v  \rangle_{\R^{n}} \,dt \\
                          + \frac12  \D_{T-t} df(T-t,.) ( \textbf{W}_{0,t}^{T}v) \,dt  \\
-  \frac12 \langle  (\widetilde{df} (T-t , U^{T}_{t})),(U^{T}_{0})^{-1}(\tpar^{T}_{0,t})^{-1} ( \Ric_{g(T-t)} - \partial_{t}(g(T-t))^{\# g(T-t)}(\textbf{W}_{0,t}^{T})v \,dt\rangle_{\R^{n}} .
\end{array}$$
By the fact that $U^{T}_{t} $ is a $g(T-t)$-isometry we have:
$$\begin{array}{l}
 \langle (df(T-t , \partial_{t}(g(T-t))^{\# T-t}(U^{T}_{t}e_{i})))_{i=1..n} ,(U^{T}_{t})^{-1} \textbf{W}_{0,t}^{T} v  \rangle_{\R^{n}} \\
 = \langle \sum_{i} \partial_{t}(g(T-t))(U^{T}_{t}e_{i} ,\nabla^{T-t}f(T-t , .) ) e_{i} , (U^{T}_{t})^{-1} \textbf{W}_{0,t}^{T}v  \rangle_{\R^{n}} \\
 = \langle \sum_{i} \partial_{t}(g(T-t))(U^{T}_{t}e_{i} ,\nabla^{T-t}f(T-t , .) ) U^{T}_{t}e_{i} , \textbf{W}_{0,t}^{T} v  \rangle_{T-t} \\
 = \langle \partial_{t}(g(T-t))^{\# T-t} (\textbf{W}_{0,t}^{T}v) , \nabla^{T-t}f(T-t , .) \rangle_{T-t} ,
 \end{array}$$
Consequently:
\begin{align*}
& d(df(T-t , . )_{X_{t}^{T}(x)} (\textbf{W}_{0,t}^{T}v))  \\                        
 &                         \overset{d \cal M }{\equiv}  -\frac{d}{dt} df(T-t,.)(\textbf{W}_{0,t}^{T}v ) \,dt \\
 &                         - \frac12 \langle \nabla^{T-t}f(T-t , .) , \partial_{t}(g(T-t))^{\# T-t}(\textbf{W}_{0,t}^{T}v)   \rangle_{
                          T-t} \,dt \\
 &                         + \frac12  \D_{T-t} df(T-t,.) (  \textbf{W}_{0,t}^{T} v) \,dt  \\
 &  - \frac12 \langle  (\widetilde{df} (T-t , U^{T}_{t})),(U^{T}_{t})^{-1} ( \Ric_{g(T-t)} - \partial_{t}(g(T-t)))^{\# g(T-t)}(\textbf{W}_{0,t}^{T})v \,dt\rangle_{\R^{n}} \\
 &      \overset{d \cal M }{\equiv}  -\frac{d}{dt} df(T-t,.)(\textbf{W}_{0,t}^{T}v ) \,dt      + \frac12  \D_{T-t} df(T-t,.) (  \textbf{W}_{0,t}^{T} v) \,dt  \\
 &- \frac12   df (T-t , \Ric_{g(T-t)}^{\# g(T-t)}(\textbf{W}_{0,t}^{T}v) \,dt. 
\end{align*}
But recall that $f$ is a solution of:
$$\frac{\partial}{\partial t}f  = \frac12 \D_{t} f  ,$$
so that
 $$ - \frac{\partial}{\partial t}df(T-t ,.) = - \frac12 d \D_{T-t} f(T-t , .). $$ 
 We shall use the Hodge-de Rham Laplacian $\square_{T-t} = -(d \delta_{T-t} + \delta_{T-t} d) $ 
which commutes with the de Rham differential, and we shall use the well-known Weitzenb\"ock formula 
(\cite{Jost,Rjost}), 
which says that for $\theta $ a 1-form, $\square_{T-t} \theta = \D_{T-t} \theta - \Ric _{T-t}\theta  $.
We get:
 $$\begin{array}{lcl}
d \D_{T-t} f(T-t , .) &=& d \square_{T-t} f(T-t , .) \\
                      &=& \square_{T-t} d f(T-t , .) \\
                      &=& \D_{T-t} d f(T-t , .) - \Ric_{T-t}  d f(T-t , .) .\\
\end{array}$$
Finally: 
 $$\begin{array}{lcl}
 d(df(T-t , . )_{X_{t}^{T}(x)} (\textbf{W}_{0,t}^{T}v))                       
                          &\overset{d \cal M }{\equiv}& \frac12 \Ric_{T-t} df(T-t,.) (\textbf{W}_{0,t}^{T}v) \,dt \\
                          & & -\frac12 \langle \nabla^{T-t}f(T-t , .) ,\Ric_{T-t}^{\# T-t}(\textbf{W}_{0,t}^{T}v)   \rangle_{
                          T-t} \,dt \\
                          &\overset{d \cal M }{\equiv}& 0,

\end{array}$$
by duality; for a 1-form $\theta$ and for $v \in TM$:
$$ \Ric(\theta) (v) = \Ric(\theta^{\#}, v.)$$
where$ \langle \theta^{\#}, v \rangle = \theta (v)$ .
\end{proof}
%

\begin{remark} \label{remarque01}
For the forward Ricci flow, we have:
  $$  \tpar^{T}_{0,t}*d((\tpar^{T}_{0,t})^{-1}\textbf{W}^{T}_{0,t}) = 0 .$$

For the backward Ricci flow, we have:
  $$  \tpar^{T}_{0,t}*d((\tpar^{T}_{0,t})^{-1}\textbf{W}^{T}_{0,t}) = - \Ric_{T-t}^{\# T-t}( \textbf{W}^{T}_{0,t}) \,dt .$$

When the family of metrics is constant, we have the usual damped parallel transport, wich satifies:
 $$  \tpar_{0,t}*d((\tpar_{0,t})^{-1}\textbf{W}_{0,t}) = - \frac12 \Ric^{\#}( \textbf{W}_{0,t}) \,dt .$$
\end{remark}

\begin{remark}
  Roughly speaking, the result says that the deformation of the metric under Ricci flow makes the damped parallel transport behaves like the damped parallel transport in the case of a constant metric with flat Ricci curvature.   
\end{remark}

For the heat equation under the forward Ricci flow, we take the probabilistic convention:
\begin{equation}\label{heatforwardRicci}
  \left\{\begin{array}{l} 
  \partial_{t}f(t,x) = \frac12 \D_{t} f(t,x) \\
  \frac{d}{dt} g_{i,j} = -\Ric_{i,j}\\
  f(0,x)=f_{0}(x) 
   \end{array} \right .
\end{equation}

We shall give a Bismut type formula and a gradient estimate formula for the above equation. For notation, let $T_{c}$ be the maximal life time of the forward Ricci flow $g(t)_{t \in[0,T_{c}[}$, solution of (\ref{forwardRicci}). For $T < T_{c} $, $X^{T}_{t}$ is a $g(T-t)$-Brownian motion and $\tpar^T_{0,t}$ the associated parallel transport.
In this case, for a solution $f(t,.)$ of (\ref{heatforwardRicci}), $ f(T-t,X^{T}_{t}(x))$ is a local martingale for any $x \in M$. 
When going back in time, one has to remember all deformations of the geometry.

We now recall a well known lemma giving a Bismut type formula (e.g.  \cite{EJL}).
Let $f(t,.)$ and $g(t)$ be solution of (\ref{heatforwardRicci}), $T<T_{c}$, and $X_{t}^{T}(x) $ a $g(T-t)$-Brownian motion.

\begin{lemma}\label{lemme}
 For all $\R^{n}$-valued process $ k$ such that $k \in L^{2}_{loc}(W)$ where $W$ is some $\R^{n} $-valued Brownian motion, and for all $v \in T_{x}M   $,
 
 $$ \begin{array}{lcl}
 N_{t} &=&  df(T-t , . )_{X_{t}^{T}(x)} (U^{T}_{t})[(U^{T}_{0})^{-1}v- \int_{0}^{t} k_{r} dr]\\        &+& f(T-t,X_{t}^{T}(x)) \int_{0}^{t} \langle k_{r} , dW \rangle_{\R^{n}}
   \end{array}$$
 is a local martingale.
\end{lemma}

\begin{proof}
The first remark after theorem \ref{damp} yield that the first term is a semi-martingale.
By It\^{o} calculus we get:
 $$d(f(T-t,X_{t}^{T}(x)) )= df(T-t,.)_{X_{t}^{T}(x)} U_{t}e_{i}dW^{i}. $$
\\
With $ (l_{i})_{i=1..n}$ a $g(T)$-orthonormal frame of $T_{x}M $, we write $N_{t}$ as:
$$ \begin{array}{lcl}
N_{t} &=& \sum_{i} (df(T-t , . )_{X_{t}^{T}(x)} (U^{T}_{t} (U^{T}_{0})^{-1})l_{i})( v_{i} -\int_{0}^{t} \langle U^{T}_{0}(k_{r}) , l_{i} \rangle_{T} dr)\\
       &+&f(T-t,X_{t}^{T}(x)) \int_{0}^{t} \langle k_{r} , dW \rangle_{\R^{n}} \\
   \end{array}$$ 
with  \ref{damp}:
$$\begin{array}{lcl}
                       
  d N_{t} &\overset{d \cal M }{\equiv}& \sum_{i} (df(T-t , . )_{X_{t}^{T}(x)} (U^{T}_{t} (U^{T}_{0})^{-1})l_{i}) (-\langle U^{T}_{0}(k_{t}) , l_{i} \rangle_{T} dt) \\
           && + d( f(T-t,X_{t}^{T}(x)))   \langle k_{t} , dW \rangle_{\R^{n}} \\
          &\overset{d \cal M }{\equiv}&  \sum_{i} (df(T-t , . )_{X_{t}^{T}(x)} (U^{T}_{t} (U^{T}_{0})^{-1})l_{i}) (-\langle U^{T}_{0}(k_{t}) , l_{i} \rangle_{T} dt) \\
          &&+ \sum_{i} df(T-t,.)_{X_{t}^{T}(x)} ( U^{T}_{t}l_{i})dW^{i}(\sum_{j}k_{t}^{j}dW^{j})\\
          &\overset{d \cal M }{\equiv}& 0.
\end{array}$$
\end{proof}

\begin{remark}
Since $T$ is smaller than the explosion time $T_{c}$, and by the compactness of $M$, $N_{t}$ is clearly a true martingale, so we could use the martingale property for global estimate, or the Doob optional sampling theorem for local estimate (e.g.~\cite{thalwang}).
\end{remark}

\begin{corolary}
Let $v \in T_{x}M$, and take for example $ k_{r} = \frac{(U^{T}_{0})^{-1}v}{T}\1_{[0,T]}(r)$ then:
 $$df(T,.)_{x}v = \frac{1}{T}\sum_{i} \E [f_{0}(X_{T}^{T}(x)) \langle (U^{T}_{0})^{-1}v, e_{i} \rangle_{\R^{n}} W_{i}(T)] .$$ 
\end{corolary}

\begin{proof}
With the above remark, $N_{t}$ is a martingale. The choice of $k_{r} $ gives 
$(U^{T}_{0})^{-1}v- \int_{0}^{T} k_{r} dr = 0$; the result follows by taking expectation at time $ 0$ and $T$. 
\end{proof}

We can give the following estimate for the gradient of the solution of (\ref{heatforwardRicci}): 

\begin{corolary}
 Let $ \| f \|_{\infty} = \sup_{M}|f_{0}| $. For $T<T_{c}$:
 $$ \sup_{x \in M} \Vert \nabla^{T}f(T,x) \Vert_{T} \text{ is decreasing in time }$$ 
 and:
 $$\sup_{x \in M} \Vert \nabla^{T}f(T,x) \Vert_{T} \leq \frac { \| f \|_{\infty} }{\sqrt{T}}.$$
\end{corolary}

\begin{proof}
Take $x \in M$ such that $\parallel\nabla^{T}f(T,x) \parallel_{T}$ is maximal. Using the
 damped parallel transport \ref{damp} we obtain that for all $v\in T_{x}M$: 
  $$df(T-t,X_{t}^{T}(x)) \textbf{W}_{0,t}^{T}v ,$$
is a local martingale.
By compactness, this is a true martingale.
Taking $v =\nabla^{T}f(T,x) $ and averaging the previous martingale at time $0$ and $t$ we get:
 $$\parallel\nabla^{T}f(T,x) \parallel_{T}^{2} = \E[ \langle \nabla^{T-t}f(T-t,X_{t}^{T}(x)),\textbf{W}_{0,t}^{T}v \rangle_{T-t}  ].$$
Using \ref{damp}, we get the first result.

If we choose $ k_{r} = \frac{(U^{T}_{0})^{-1}v}{T}\mathbbm{1}_{[0,T]}(r)$ in \ref{lemme}, then $N_{t} $ is a martingale. Taking expectations at times $0$ and $T$, we obtain 
$$df(T,.)_{x}v = \frac1T \E [f_{0}(X_{T}^{T}(x)) \int_{0}^{T}\langle U^{T}_{0})^{-1}v, dW \rangle_{\R^{n}} ]. $$
For $x\in M$ and $v = \nabla^{T}f(T,x)  $, Schwartz inequality gives
$$\parallel\nabla^{T}f(T,x) \parallel_{T}^{2} \le \frac{\mathcal{M}_{0}}{T} \E \left[\left| \int_{0}^{T} \langle U^{T}_{0})^{-1}v, dW \rangle_{\R^{n}}\right|^{2} \right]^{\frac12}. $$
We have: 
$$ \E \left[\left| \int_{0}^{T} \langle U^{T}_{0})^{-1}v, dW \rangle_{\R^{n}}\right|^{2} \right] = T \parallel v \parallel_{T}^2. $$
The result follows.
\end{proof}

For geometric interpretation, let us give an example of normalized Ricci flow for surfaces (which is completely understood e.g. \cite{chowknopf}). We are interested in this example because the equation for the scalar curvature under this flow is a reaction-diffusion equation which is quite similar to the heat equation under Ricci flow. We will give a gradient estimate formula for the scalar curvature under normalized Ricci flow which gives in the case $ \chi(M) < 0$ (the easiest case) the convergence of the metric to a metric of constant curvature.

The normalized Ricci flow of surfaces comes from normalizing the metric by some time dependent function to preserve the volume. Let $M$ be a 2-dimensional manifold, $R(t) $ the scalar curvature, 
$r= {\int_{M} R_{t} d\mu_{t}}/{\mu_{t}(M)}$ its average (which will be constant in time, as topological constant, e.g. Gauss-Bonnet). We get the following equation for normalized Ricci flow: 
 $$ \frac{d}{dt} g_{i,j}(t)  = (r-R(t)) g_{i,j}(t).$$

\begin{remark}
Hamilton gives a proof of the existence of solutions to this equation, defined for all time ( e.g. \cite{chowknopf})). 
\end{remark}

Recall that  (e.g. \cite{chowknopf}) the equation for the scalar curvature $R$ is:
$$ \frac{\partial}{\partial t} R = \D_{t}R + R(R-r). $$

\begin{proposition}\label{prop-scalar}
Let $T \in \R $, $X_{t}^{T}(x) $ be a $\frac12 g(T-t)$-BM$(x)$, $\tpar^{T}_{0,t} $ the parallel transport, $v \in T_{x}M $ and $\phi_{t}v $ the solution of the following equation:
  $$ \tpar^{T}_{0,t} d \left((\tpar^{T}_{0,t})^{-1} \phi_{t}v\right) = -\left(\frac32 r -2R\left(T-t,X_{t}^{T}(x)\right)\right)\phi_{t}v  \,dt $$
$$ \phi_{0}=\Id_{T_{x}M}.$$
Then $dR(T-t,.)_{X_{t}^{T}(x)} \phi_{t}v  $ is a martingale and:
\begin{equation} \label{commetuveux} \Vert \nabla^{T}R(T,x) \Vert_{T} \leq \sup_{M}\Vert \nabla^{0}R(0,x) \Vert_{0} e^{-\frac32 rT} \mathbb{E}[e^{\int_{0}^{T} 2R(T-t,X_{t}^{T}(x))\,dt}]. \end{equation}
\end{proposition}

\begin{proof}
 The proof is similar to the one in \ref{damp}, the difference is the reaction term: $ R(R-r)$. 
For notations and some details see the proof of \ref{damp}.
Take $F : x \mapsto x(x-r) $, then:
$$ \frac{\partial}{\partial t} R = \D_{t}R + F(R). $$
We write:
$$dR(T-t, .)\mid_{X_{t}^{T}(x)}\phi_{t}v = \langle \tilde dR(T-t,U_{t}^{T}),(U_{t}^{T})^{-1}\phi_{t}v \rangle_{\R^{2}} $$
where $U_{t}^{T}$ is a diffusion on $\mathcal{F}(M)$ with generator 
$$\D_{T-t}^{H} + \frac14 (r-R(T-t, \pi .))g(T-t)(\ev_{ei} . ,\ev_{ej} .)  V_{i,j}(.)  .$$ 
Using theorem \ref{damp}, we have:
\begin{align*}
\displaystyle  & d\langle \tilde dR(T-t,U_{t}^{T}),(U_{t}^{T})^{-1}\phi_{t}v \rangle_{\R^{2}}\\
& = \langle d(\tilde dR(T-t,U_{t}^{T})),(U_{t}^{T})^{-1}\phi_{t}v \rangle_{\R^{2}} \\
&+ \langle \tilde dR(T-t,U_{t}^{T}),d((U_{t}^{T})^{-1}\phi_{t}v) \rangle_{\R^{2}} \\
&\overset{d \cal M }{\equiv}   \Big[\frac{\partial}{\partial_{t}} (dR(T-t,.) )+ \D_{T-t}dR(t-t,.) +\frac12 (r-R(T-t,\pi .))dR(T-t,.)\Big](\phi_{t}v) \,dt  \\
&+ \langle \tilde dR(T-t,U_{t}^{T}),d((U_{t}^{T})^{-1}\phi_{t}v) \rangle_{\R^{2}} 
\end{align*}
Using Weitzenb\"ock formula and the equation for $R$ we get:
$$\frac{\partial}{\partial t} dR(T-t,.) = -[\D_{T-t}dR(T-t,.) - \Ric_{T-t}dR(T-t,.) + F^{'}(R(T-t,.))dR(T-t,.) ] $$
Recall that for the surface:
$$ \Ric_{T-t}dR(T-t,.) = \frac12 R(T-t,.)dR(T-t,.) , $$
consequently
\begin{align*}
& d\langle \tilde dR(T-t,U_{t}^{T}),(U_{t}^{T})^{-1}\phi_{t}v \rangle_{\R^{2}}\\
& \overset{d \cal M }{\equiv} ( \frac12  r - F^{'}(R(T-t,.))dR(T-t,.))(\phi_{t}v)\,dt +\langle \tilde dR(T-t,U_{t}^{T}),d((U_{t}^{T})^{-1}\phi_{t}v) \rangle_{\R^{2}} \\
&\overset{d \cal M }{\equiv}  
( \frac12  r - 2R(T-t,.) + r)dR(T-t,.))(\phi_{t}v)\,dt \\
&+\langle \tilde dR(T-t,U_{t}^{T}),(U_{t}^{T})^{-1}(-\frac32 r + 2R(T-t,.))\phi_{t}v) \rangle_{\R^{2}} \\
&\overset{d \cal M }{\equiv}  0,
\end{align*}
where we used the equation of $\phi_{t}v $ in the last step.

 For the second part of the proposition, with the equation for $ \phi_{t}v $ we have:
$$ d (\parallel\phi_{t}v \parallel_{T-t}^{2} )= (4R(T-t, X_{t}^{T}(x)-3r )\parallel\phi_{t}v \parallel_{T-t}^{2} \,dt,$$
so that
$$\parallel\phi_{T}v \parallel_{0}^{2} = \parallel\phi_{0}v \parallel_{T}^{2} e^{-3rT} e^{\int_{0}^{T} 4R(T-s,X_{s}^{T}(x))\,ds}. $$

Take $v= \nabla_{T}R(T,x) $ and average at time $0$ and $T$ (it is a true martingale because all coefficients are bounded) to get:
$$\Vert \nabla^{T}R(T,x) \Vert_{T} \leq \sup_{M}\Vert \nabla^{0}R(0,x) \Vert_{0} e^{-\frac32 rT} \mathbb{E}[e^{\int_{0}^{T} 2R(T-s,X_{s}^{T}(x))\,ds}].$$
\end{proof}

\begin{remark}
 For reaction-diffusion equations we can find by this calculation the correction to the parallel transport leading to a Bismut type formula for the gradient of the equation:
\begin{equation}
  \frac{\partial}{\partial t} f = \D_{t}f + F(f) ,\label{recdif}
\end{equation}
 where $\D_{t} $ is a Laplace Beltrami operator associated to a family of metrics $g(t)$. Let $X_{t}^{T}(x) $ be a $\frac12 g(T-t)-BM(x) $, $\tpar^{T}_{0,t} $ the associated parallel transport    and $v \in T_{x}M $. Consider the covariant equation:
$$\tpar^{T}_{0,t}d (\tpar^{T}_{0,t})^{-1} \Theta_{t}v = -\Big( \Ric^{\#,T-t}-\frac12 \Big[\frac{\partial}{\partial t}( g(T-t))\Big]^{\#,T-t} - F^{'}(f) \Big)\Theta_{t}v \,dt $$
Then for $f$ a solution of (\ref{recdif}) and  $v \in T_{x}M $ we obtain that: 
$$df(T-t, .) \Theta_{t}v $$
is a local martingale.
\end{remark}

\begin{corolary}
For $\chi(M)<0 $, there exists $C>0 $ depending only on $ g(0)$, such that:
$$\Vert \nabla^{T}R(T,x) \Vert_{T} \leq \sup_{M}\Vert \nabla^{0}R(0,x) \Vert_{0}\, e^{\frac12 rT}e^{2C(\frac{e^{rT}-1}{r})} . $$
\end{corolary}

\begin{proof} 
We use proposition 5.18 in \cite{chowknopf}. In this case we have $ r<0 $ and a constant $C >0$ depending only on the initial metric  such that $R(t,.) \le r+Ce^{rt}$ and the estimate follows from previous proposition.
\end{proof}

\begin{remark}
 For the case $\chi(M)<0 $ we obtained an estimate which decreases exponentially.
For the case $\chi(M)>0$ one could control the expectation in (\ref{commetuveux}).
\end{remark}
%

\section{The point of view of the stochastic flow}

Let $g(t)_{[0,T_{c}[}$ be a $C^{1,2}$ family of metrics, and consider the heat equation:
\begin{equation}
  \left\{\begin{array}{l} 
  \partial_{t}f(t,x) = \frac12 \D_{t} f(t,x) \\
   f(0,x)=f_{0}(x),
   \end{array} \right .
\end{equation}
where $f_{0}$ is a function over $M$. We suppose that the solution of this equation exists until $T_{c}$. For $T < T_{c} $, let $X^{T}_{t}$ be a $g(T-t)$-Brownian motion and $\tpar^T_{0,t}$ the associated parallel transport.

We will build (c.f. (\ref{eqflow})) a family of semimartingales $ (T-t,X^{T}_{t}(x))  $ such as $X^{T}_{t}(x)$ is a $g(T-t)$-BM($x$) for all $x$ nearby $x_{0}$ and such that the family of martingales $ f(T-t,X^{T}_{t}(x))_{x}$ 
is differentiable at $x_{0}$ with respect to the parameter $x$. However, in  this section, we will not do it directly using stochastic flows in the sense of \cite{kunita}. 
Instead, we will use differentiation of families of martingales defined as limit in some semi-martingale space 
(the topology is as in \cite{Etop} which has been extended by Arnaudon, Thalmaier to the manifold case \cite{MThor}, \cite{MTstab}, 
\cite{MTcomp}, \cite{MTym}).

We work in the space-time $ I\times M$, its tangent bundle being identified to $ TI \times TM $ endowed with the cross connection $ \tilde \nabla = \overline{\nabla} \otimes \nabla_{T-t} $ where $\overline{\nabla} $ is the flat connection. Let $X_{t}^{T}(x_{0}) $ be a $g(T-t)$-BM started at $x_{0}$, and define $ Y_{t}(x_{0}) = (t, X_{t}^{T}(x_{0}))$ a $I \times M  $-valued semi martingale.
From now on $P^{\tilde \nabla }_{X,Y}$ stands for the parallel transport along the shortest $\tilde \nabla$-geodesic between nearby points $ X \in I\times M  $ and $ Y \in I\times M  $ for the connection $\tilde \nabla$.

Let $\tilde c $ a curve in $ I\times M$, we write $P^{\tilde \nabla }_{\tilde c}$ for the $\tilde \nabla$ parallel transport along $ \tilde c$ and for a curve $c$  in $M$ we denote by $\tpar_{c}^{T-s}$ the $\nabla^{T-s}$ parallel transport along $c$. We also denote
$\pi : I \times M \rightarrow M$ the natural projection.

For a curve $\gamma : t \longrightarrow (s,x_{t})$ in $I\times M  $, where  $s$ is a fixed time, we have the following observation:
 $$P^{\tilde \nabla }_{\gamma} = (\Id , \tpar_{\pi (\gamma)}^{T-s}).$$
 Define the It\^{o} stochastic equation in the sense of \cite{Emery-transfer}:
 \begin{equation} \label{eqflow}
  d^{\tilde \nabla } Y_{t}(x) = P^{\tilde \nabla }_{Y_{t}(x_{0}),Y_{t}(x)} d^{\tilde \nabla } Y_{t}(x_{0}) 
 \end{equation}
\begin{remark}
The above equation is well defined, for $x$ sufficiently close to $x_{0}$, because $d_{T-t}(X_{t}(x),X_{t}(x_{0})) $ is a finite variation process, with bounded derivative (by a short computation and  \cite{kendall}, \cite{cranston}).
\end{remark}

Let $ \tilde\tpar_{0,t} $  be the parallel transport, associated to the connection $\tilde \nabla$, over the semi martingale $Y_{t}(x_{0}) $.

In the next lemma, we will explain the relationship between the two parallel transport $ \tilde\tpar_{0,t} $ and $\tpar^T_{0,t}$.

\begin{lemma}\label{lemme-relation-transport}
  Let $(e_{i})_{i=1..n}$ be a orthonormale of $(T_{x_{0}}M,g(T))$ then
$$d((\tpar^{T}_{0, t})^{-1} d\pi \tilde\tpar_{0,t} )(0, e_{i})
=\frac12 (\tpar^{T}_{0, t})^{-1} (\frac{\partial}{\partial_{t}} g(T-t))^{\# T-t} (d\pi \tilde\tpar_{0,t} (0,e_{i}) )\,dt.$$
\end{lemma}
\begin{proof}
The parallel transport $\tilde\tpar_{0, t}$ does not modify the time vector, i.e.,
$$\tilde\tpar^{-1}_{(t,X_{t})}(0,...) = (0,... ) ,$$
as can be shown for every curves, and hence for the semi-martingale $Y_{t} $ by the transfer principle. 

We identify 
 $\tilde{T}=\{ (0,v) \in T_{(0,x_{0})} I \times M \}$ and  $T_{x_{0}}M $ with the help of $(0,v) \longmapsto v$.
Hence
$$(\tpar^{T}_{0, t})^{-1}  d\pi \tilde\tpar_{0, t} : \tilde{T} \to T_{x_{0}}M$$
becomes an element in $ \mathcal{M}_{n,n}(\R)$.
 
Recall that $\tpar^{T}_{0, t} = U_{t}^{T} U_{0}^{T,-1} $. By definition of $D^{S,t}$ given in (\ref{cov_mart}). We get using the shorthand $e_{i}=U_{0}^{T} \tilde e_{i}$, with $(\tilde e_{i})_{i=1..n}$ an orthonormal frame of $\R^{n} $,
\begin{eqnarray*}
 *d((\tpar^{T}_{0, t})^{-1} d\pi \tilde\tpar_{0, t})&=& *d (\langle (\tpar^{T}_{0, t})^{-1} d\pi \tilde\tpar_{0, t} e_{i},e_{j} \rangle_{T})_{i,j} \\
&=&*d (\langle  d\pi \tilde\tpar_{0,t} e_{i},\tpar^{T}_{0, t} e_{j} \rangle_{T-t})_{i,j} \\
&=&\Big( \langle D^{S,T-t}  d\pi \tilde\tpar_{0,t} e_{i} ,U_{t}^{T} \tilde{e}_{j} \rangle_{T-t} \\
& & + \frac{\partial}{\partial_{t}} (g(T-t)) (d\pi \tilde\tpar_{0,t} e_{i} ,U_{t}^{T} \tilde{e}_{j})\,dt \\
& &+ \langle   d\pi \tilde\tpar_{0,t} e_{i} ,D^{S,T-t} U_{t}^{T} \tilde{e}_{j} \rangle_{T-t} \Big)_{i,j}. \\
\end{eqnarray*} 
We also have:
$$\begin{array} {lcl}
  D^{S,T-t}  d\pi \tilde\tpar_{0,t} e_{i} &=& \mathcal{V}_{ d\pi \tilde\tpar_{0,t} e_{i}}^{-1} ( (*d ( d\pi \tilde\tpar_{0,t} e_{i}))^{v_{T-t}}) \\
 &=& \mathcal{V}_{d\pi \tilde\tpar_{0,t} e_{i}}^{-1} ( (dd\pi d\ev_{e_{i}} (*d\tilde\tpar_{0,t}) )^{v_{T-t}}) \\ 
 &=& \mathcal{V}_{d\pi \tilde\tpar_{0,t} e_{i}}^{-1} ( dd\pi (d\ev_{e_{i}} (*d\tilde\tpar_{0,t}) )^{\tilde v}) \\  
&=& 0 .
\end{array}$$
Where we have used in the last equality the fact that $\tilde \tpar_{0,t}  $ is the $\tilde \nabla  $ horizontal lift of $Y_{t} $. The third one may be seen for curves, it comes from the definition of $\tilde \nabla  $. 

Following computations similar to one in the first section, we have by (\ref{Ddef}):
$$\begin{array} {lcl}
 *d((\tpar^{T}_{0, t})^{-1} d\pi \tilde\tpar_{0,t} )_{i,j} &=&\frac{\partial}{\partial_{t}} g(T-t) (d\pi \tilde\tpar_{0,t} e_{i} ,U_{t}^{T} \tilde{e}_{j})\,dt \\
& +& \langle   d\pi \tilde\tpar_{0,t} e_{i} ,D^{S,T-t} U_{t}^{T} \tilde{e}_{j} \rangle_{T-t} \\
& =& \frac{\partial}{\partial_{t}} g(T-t) (d\pi \tilde\tpar_{0,t} e_{i} ,U_{t}^{T} \tilde{e}_{j})\,dt \\
& +& \langle   d\pi \tilde\tpar_{0,t} e_{i}, - \frac12 \sum_{\alpha=1}^{d}\frac{\partial}{\partial_{t}} g(T-t)(U_{t}^{T} \tilde{e}_{j},U_{t}^{T} \tilde{e}_{\alpha} )U_{t}^{T} \tilde{e}_{\alpha} \rangle_{T-t} \,dt \\
&=& \frac{\partial}{\partial_{t}} g(T-t) (d\pi \tilde\tpar_{0,t} e_{i} ,U_{t}^{T} \tilde{e}_{j})\,dt \\
&-&  \frac12 \sum_{\alpha=1}^{d}\frac{\partial}{\partial_{t}} g(T-t)(U_{t}^{T} \tilde{e}_{j},U_{t}^{T} \tilde{e}_{\alpha} )\langle   d\pi \tilde\tpar_{0,t} e_{i},U_{t}^{T} \tilde{e}_{\alpha} \rangle_{T-t} \,dt\\
&=& \frac12 \frac{\partial}{\partial_{t}} g(T-t) (d\pi \tilde\tpar_{0,t} e_{i} ,U_{t}^{T} \tilde{e}_{j})\,dt .\\
\end{array}$$

In the general case, and by previous identification:
\begin{eqnarray} \label{partpi}
d((\tpar^{T}_{0, t})^{-1} d\pi \tilde\tpar_{0,t} )(0, e_{i})
&=&\frac12 \sum_{j} \frac{\partial}{\partial_{t}} g(T-t) (d\pi \tilde\tpar_{0,t} e_{i} ,U_{t}^{T} \tilde{e}_{j}) e_{j}\,dt\\
&=&\frac12 (\tpar^{T}_{0, t})^{-1} (\frac{\partial}{\partial_{t}} g(T-t))^{\# T-t} (d\pi \tilde\tpar_{0,t} (0,e_{i}) )\,dt.
\end{eqnarray}.
 
\end{proof}

 Differentiating (\ref{eqflow}) along a geodesic curve beginning at $(0,  x_{0})$ with velocity $(a,v) \in T_{0} I\times T_{x_{0}}M$ and using corollary 3.17 in \cite{MTstab} we get:
$$\tilde\tpar_{0,t} d \big( \tilde\tpar^{-1}_{0,t} TY_{t}(a,v) \big) = - \frac12 \tilde R (TY_{t}(a,v), dY_{t}(x_{0}))dY_{t}(x_{0}) , $$
where $\tilde R  $ is the curvature tensor. 

Let $v \in T_{x}M $ we write:
$$TX_{t} v := d\pi TY_{t}(0,v) .$$


In a more canonical way than theorem \ref{damp}, we have the following proposition.

\begin{proposition}\label{stoc_flow}
 For all $ v \in T_{x}M$ we have:
  $$ d ((\tpar^{T}_{0,t})^{-1}TX_{t} v) =  \frac12 (\tpar^{T}_{0,t})^{-1} ( (\frac{\partial}{\partial_{t}} g(T-t)) -  \Ric_{T-t})^{\#,T-t}(TX_{t} v) \,dt .\\  $$
\end{proposition}

\begin{proof}
 For a triple of tangent vectors $ (L_{t},L),(A_{t},A),(Z_{t},Z) \in TI \times TM$, we have:
$$\tilde R ((L_{t},L),(A_{t},A))(Z_{t},Z) = (0 , R_{T-t}(L,A)Z). $$ 
Hence, according to the relation $dY(x_{0}) = (dt,*dX_{t})=(dt, \tpar^{T}_{0,t}e_{i}*dW^{i})$ and the definition of the Ricci tensor:
\begin{equation} \label{floY}
\tilde\tpar_{0,t}  d \big{(}\tilde\tpar^{-1}_{0,t} TY_{t}(0,v) \big{)}= - \frac12 (0, \Ric^{\#T-t}( TX_{t} v )) \,dt .
\end{equation}
In order to compute in $\R^{n} $, we write:
\begin{equation} \label{floRic}
 (\tpar^{T}_{0, t})^{-1} TX_{t} v = ((\tpar^{T}_{0, t})^{-1} d\pi \tilde\tpar_{0,t}) ( \tilde\tpar^{-1}_{0,t} TY_{t}(0,v)).  
\end{equation}

By (\ref{floY}), we have $ d \big{(} \tilde\tpar^{-1}_{0,t} TY_{t}(0,v)\big{)}  \in d\mathcal{A} $ where $\mathcal{A} $ is the space of finite variation processes.
We get:
\begin{equation*} 
 d( (\tpar^{T}_{0, t})^{-1}TX_{t} v) =d((\tpar^{T}_{0, t})^{-1} d\pi \tilde\tpar_{0,t}) ( \tilde\tpar^{-1}_{0,t} TY_{t}(0,v)) + ((\tpar^{T}_{0, t})^{-1} d\pi \tilde\tpar_{0,t})  d ( \tilde\tpar^{-1}_{0,t} TY_{t}(0,v)) .\\
\end{equation*}

By (\ref{floRic}) and lemma \ref{lemme-relation-transport} we get:
$$\begin{array} {lcl}
 d((\tpar^{T}_{0, t})^{-1} TX_{t} v) &=& *d((\tpar^{T}_{0, t})^{-1} d\pi \tilde\tpar_{0,t})  ( \tilde\tpar^{-1}_{0,t} TY_{t}(0,v))\\
&+& ((\tpar^{T}_{0, t})^{-1} d\pi \tilde\tpar_{0,t})  *d ( \tilde\tpar^{-1}_{0,t} TY_{t}(0,v)) \\
&=&*d((\tpar^{T}_{0, t})^{-1} d\pi \tilde\tpar_{0,t})  ( \tilde\tpar^{-1}_{0,t} TY_{t}(0,v)) \\
&-& \frac12 ((\tpar^{T}_{0, t})^{-1} d\pi )(0,\Ric^{\#,T-t}( TX_{t} v)\,dt \\
&=& \frac12 (\tpar^{T}_{0, t})^{-1} (\frac{\partial}{\partial_{t}} g(T-t))^{\# T-t} (TX_{t} v) \,dt\\
&-&  \frac12 (\tpar^{T}_{0, t})^{-1} \Ric^{\#,T-t}(TX_{t} v) \,dt.\\
\end{array}$$
\end{proof}
%
 
For all $f_{0} \in C^{\infty}(M)$, and for $f(t, .)$ a solution of \eqref{heatforwardRicci}, $f(T-t,X_{t}^{T}(x)) $ is a martingale, where  $ (T-t,X_{t}^{T}(x)) = Y_{t}(x) $ is built as in \eqref{eqflow}. We have the following corollary which agrees with theorem \ref{damp}.

\begin{corolary} \label{coroflow}
 For all $v \in T_{x}M $:
$$ d f(T-t,X_{t}^{T}(x)) v =  df(T-t, . )_{X_{t}^{T}(x)} \tpar^{T}_{0,t} v  , $$
is a martingale.
\end{corolary}

\begin{proof}
 By differentiation under $x$ of $f(T-t,X_{t}^{T}(x)) $, we get a local martingale. According to \cite{MTstab} and by chain rule for differential we get the corollary. 
 This result matches \ref{damp} after using the above proposition.
\end{proof}

In an canonical way, we have the following result.
\begin{stheorem}\label{th}
The following conditions are equivalent for a family $g(t)$ of metrics:
\begin{enumerate}[i)]
\item  $g(t)$ evolves under the forward Ricci flow.
\item For all $T < T_{c}$ we have $\tpar^{T}_{0,t} = \textbf{W}_{0,t}^{T}=TX_{t}$.
\item For all $T < T_{c}$, the damped parallel transport $\textbf{W}_{0,t}^{T}$ is an isometry. 
\end{enumerate}
\end{stheorem}

\begin{proof}
Here, the forward Ricci flow has probabilistic convention \eqref{forwardRicci}. The result follows by the equation of $ g(t)$ and by proposition \ref{stoc_flow} and theorem \ref{damp}.
\end{proof}

\section{Second derivative of the stochastic flow}

We take the differential of the stochastic flow in order to obtain a intrinsic martingale.
We take the same notation as the previous section, and $g(t)$ is a family of metrics coming from a forward Ricci flow. Let $X_{t}^{T}(x) $ be the $g(T-t)$-BM started at $x$, constructed as in the previous section by the parallel coupling of a $g(T-t)$-BM started at $x_{0}$ ( \ref{eqflow} ), $\tilde\nabla$ and $ Y_{t}(x) = (t, X_{t}^{T}(x))$ as before, define the intrinsic trace (that do not depend on the choice of $E_{i}$ as below):
$$\Tr \nabla_{.}TX_{t}(x_{0}) (.) := d \pi  \bigg(\sum_{i}  \tilde\nabla_{(0,e_{i})} TY_{t}(x)(0,E_{i}(x))  - TY_{t}(x) \tilde\nabla_{(0,e_{i})}(0,E_{i}(x))   \bigg)$$
where $(e_{i}) $ is a $ (T_{x_{0}}M, g(T))$ orthonormal basis, $E_{i}$ are vectors fields in $\Gamma TM $ such that $E_{i}(x_{0})=e_{i}$ and $ \tilde\nabla_{(0,e_{i})} TY_{t}(x)(0,E_{i}(x))$ is a derivative of a bundle-valued semi-martingale in the sense of (\cite{MThor}, \cite{MTstab}, \cite{MTcomp}).
By \ref{th}: 
 $$\Tr \nabla_{.}TX_{t}(x_{0}) (.) := d \pi  \sum_{i}  \tilde\nabla_{(0,e_{i})} TY_{t}(x)(0,E_{i}(x)) - \tpar^{T}_{0,t} d\pi (\sum_{i} \tilde\nabla_{(0,e_{i})}(0,E_{i}(x)) ) $$

\begin{stheorem}\label{mart-can} Let 
 $L_{t}:=(\tpar^{T}_{0,t})^{-1} \Tr \nabla_{.}TX_{t}(x_{0}) (.) $ be a $(T_{x_{0}}M,g(T)) $-valued process, started at $0$. 
Then:
\begin{enumerate}[i)] 
\item $L_{t}$ is a $(T_{x_{0}}M,g(T)) $-valued martingale, independent of the choice of $E_{i}$. 
\item The $g(T)$-quadratic variation of $L$ is given by $d[L,L]_{t} = \parallel \Ric^{T-t}(X_{t}(x_{0}))\parallel_{T-t}^{2} \,dt $.
\end{enumerate}
\end{stheorem}

\begin{proof}
Recall that by the same construction of the previous section:
$$\tilde{D} (TY_{t}(x)(0,E_{i}(x))) = -\frac12 \tilde{R}(TY_{t}(x)(0,E_{i}(x)),dY_{t}(x))dY_{t}(x) .$$
By the general commutation formula (e.g. theorem 4.5 in \cite{MThor}), and by the previous equation which cancels two terms in this formula, we get:
\begin{align*}
 \tilde{D} \tilde\nabla_{(0,e_{i})}(TY_{t}(x)(0,E_{i}(x))) =&  \tilde\nabla_{(0,e_{i})} \tilde{D} (TY_{t}(x)(0,E_{i}(x))) \\
 & + \tilde{R}(d^{\tilde\nabla}Y_{t}(x_{0}),TY_{t}(x_{0})(0,e_{i})) TY_{t}(x_{0})(0,e_{i}) \\
 &  -\frac12 \tilde\nabla \tilde{R} (dY_{t}(x_{0}),TY_{t}(x_{0})(0,e_{i}),dY_{t}(x_{0})) TY_{t}(x_{0})(0,e_{i})\\
=& -\frac12  \tilde\nabla_{(0,e_{i})} (\tilde{R}(TY_{t}(x) (0,E_{i}(x)),dY_{t}(x))dY_{t}(x)) \\
 & + \tilde{R}(d^{\tilde\nabla}Y_{t}(x_{0}),TY_{t}(x_{0})(0,e_{i}))TY_{t}(x_{0})(0,e_{i}) \\
 & -\frac12 (\tilde\nabla_{dY_{t}(x_{0})} \tilde{R}) (TY_{t}(x_{0})(0,e_{i}),dY_{t}(x_{0})) TY_{t}(x_{0})(0,e_{i}).\\
\end{align*} 
Taking  trace in the previous equation we can go one step further. Recall that $(e_{i})_{i=1..n}$ is a orthogonal basis of $(T_{x_{0}}M, g(T)) $, and write for notation:
 $$ \tilde{\Ric}^{\#}_{(t,x)} (V)= (0,{\Ric}^{\# T-t}(d\pi V)),$$
 then:
$$
\begin{array}{lcl}
&\sum_{i}& \tilde{D} \tilde\nabla_{(0,e_{i})}(TY_{t}(x)(0,E_{i}(x))) \\
&=& -\frac12 \sum_{i}\tilde\nabla_{(0,e_{i})}(\tilde{\Ric}^{\#}_{Y_{t}(x)}(TY_{t}(x)E_{i}(x)))  \\
&& + \sum_{i} \tilde{R}(d^{\tilde\nabla}Y_{t}(x_{0}),TY_{t}(x_{0})(0,e_{i}))TY_{t}(x_{0})(0,e_{i})  \\
&& -\frac12 \sum_{i} (\tilde\nabla_{dY_{t}(x_{0})} \tilde{R}) (TY_{t}(x_{0})(0,e_{i}),dY_{t}(x_{0})) TY_{t}(x_{0})(0,e_{i})\\
&=&-\frac12 \sum_{i} (\tilde\nabla_{(TY_{t}(x_{0})(0,e_{i}))} \tilde{\Ric}^{\#}) (TY_{t}(x_{0})(0,e_{i})) \\
&& -\frac12  (\tilde{\Ric}^{\#}_{Y_{t}(x_{0})}(\sum_{i} \tilde\nabla_{(0,e_{i})} TY_{t}(x)(0,E_{i}(x)))) + \tilde{\Ric}^{\#}_{Y_{t}(x_{0})} (d^{\tilde\nabla}Y_{t}(x_{0})) \\
&&  -\frac12 \sum_{i} (\tilde\nabla_{dY_{t}(x_{0})} \tilde{R}) (TY_{t}(x_{0})(0,e_{i}),dY_{t}(x_{0})) TY_{t}(x_{0})(0,e_{i}).
\end{array}$$
In the last equality, we use the chain derivative formula, and derivation is taking with respect to $x$ .
We will make an independent computation for the last term in the previous equation. Let $\Tr$ stand for the usual trace:
$$\begin{array}{l}
 \sum_{i} (\tilde\nabla_{dY_{t}(x_{0})} \tilde{R}) (TY_{t}(x_{0})(0,e_{i}),dY_{t}(x_{0})) TY_{t}(x_{0})(0,e_{i}) \\
 = \sum_{i}(0,(\nabla_{dX_{t}}^{T-t} R^{T-t}) (TX_{t}(x_{0})e_{i},dX_{t}) TX_{t}(x_{0})e_{i}) \\
= \sum_{i,j} (0,(\nabla_{\tpar_{0,t}^{T}e_{j}}^{T-t} R^{T-t}) (TX_{t}(x_{0})e_{i},\tpar_{0,t}^{T}e_{j}) TX_{t}(x_{0})e_{i})\,dt\\
= \sum_{j}(0, \Tr_{1,3} (\nabla_{\tpar_{0,t}^{T}e_{j}}^{T-t} R^{T-t}) (\tpar_{0,t}^{T}e_{j})) \,dt\\
= \sum_{j}(0,(\nabla_{\tpar_{0,t}^{T}e_{j}}^{T-t} \Tr_{1,3} R^{T-t}) (\tpar_{0,t}^{T}e_{j})) \,dt\\
= - \sum_{j}(0,(\nabla_{\tpar_{0,t}^{T}e_{j}}^{T-t} \Ric^{\# T-t})(\tpar_{0,t}^{T}e_{j})) \,dt, 
\end{array}$$
where we have used in the second equality the fact that in case of the forward Ricci flow $\tpar_{0,t}^{T}$ is a $g(T-t)$ isometry and $ dX =\tpar_{0,t}^{T}e_{j} dW^{j}$.
In the last equality we use the commutation between trace and covariant derivative (for example \cite{Lee}, or \cite{Kob-Nom}).
Note that:
$$\begin{array}{l}
 \sum_{i} (\tilde\nabla_{(TY_{t}(x_{0})(0,e_{i}))} \tilde{\Ric}^{\#}) (TY_{t}(x_{0})(0,e_{i})) \\
 = \sum_{i} (0, (\nabla_{TX_{t}(x)e_{i}}^{ T-t}\Ric_{X^{T}_{t}(x)}^{\# T-t})(TX_{t}(x)e_{i})) \,dt\\
\end{array}$$
Hence, using \ref{th}:
$$ \begin{array}{l}
    \tilde{D} ( \sum_{i} \tilde\nabla_{(0,e_{i})} TY_{t}(x)(0,E_{i}(x)))\\
\quad= -\frac12  (\tilde{\Ric}^{\#}_{Y_{t}(x_{0})}(\sum_{i} \tilde\nabla_{(0,e_{i})} TY_{t}(x)(0,E_{i}(x))))
 + \tilde{\Ric}^{\#}_{Y_{t}(x_{0})} (d^{\tilde\nabla}Y_{t}(x_{0})) 
   \end{array}$$
Write, for simplicity, $ B $ for $  \sum_{i}  \tilde\nabla_{(0,e_{i})} TY_{t}(x)(0,E_{i}(x))$.
We compute:
$$\begin{array}{rcl}
   d( \tpar_{0,t}^{T,-1} d\pi  B) &=& d ([\tpar_{0,t}^{T,-1} d\pi \tilde{\tpar}_{0,t}][(\tilde{\tpar}_{0,t})^{-1} B]) \\
 &= & \frac12 \tpar_{0,t}^{T,-1} (\partial_{t} g(T-t))^{\#,T-t}(d\pi B)\,dt  \\
 && \quad+\tpar_{0,t}^{T,-1} (- \frac12 d\pi (\tilde{\Ric}^{\#}(B))+ d\pi ( \tilde{\Ric}^{\#}_{Y_{t}(x_{0})} (d^{\tilde\nabla}Y_{t}(x_{0})))) \\
&=&\tpar_{0,t}^{T,-1}(d\pi \tilde{\Ric}^{\#}_{Y_{t}(x_{0})} (d^{\tilde\nabla}Y_{t}(x_{0})))\\
&=& \sum_{i}\tpar_{0,t}^{T,-1}\Ric_{X^{T}_{t}(x)}^{\# T-t}(\tpar^{T}_{0,t}e_{i})dW^{i},\\
  \end{array}
 $$
where we have used lemma \ref{lemme-relation-transport} in  the first equality. We get a intrinsic martingale that does not depend on $E_{i}$, starting  at $0$. By the definition in theorem \ref{mart-can} and by the formula preceding theorem \ref{mart-can}, the above calculations yield:
 $$L_{t}= \int_{0}^{t} \sum_{i}\tpar_{0,t}^{T,-1}\Ric_{X^{T}_{t}(x)}^{\# T-t}(\tpar^{T}_{0,t}e_{i})dW^{i} - d\pi (\sum_{i} \tilde\nabla_{(0,e_{i})} (0,E_{i}(x))).$$ 
For the $g(T)$-quadratic variation of $L_{t}$ we use the isometry property of the parallel transport; we compute the quadratic variation: 
$$\begin{array}{rcl}
  d[L,L]_{t} &=& \langle \tpar_{0,t}^{T,-1}\Ric_{X^{T}_{t}(x)}^{\# T-t}(\tpar^{T}_{0,t}e_{i}),\tpar_{0,t}^{T,-1}\Ric_{X^{T}_{t}(x)}^{\# T-t}(\tpar^{T}_{0,t}e_{i}) \rangle_{T} \,dt\\
&=& \sum_{i} \parallel \Ric_{X^{T}_{t}(x)}^{\# T-t}(\tpar^{T}_{0,t}e_{i}) \parallel_{g(T-t)}^{2} \,dt\\
&=& \norm \Ric_{X^{T}_{t}(x)}^{\# T-t} \norm_{T-t}^{2}\,dt;\\
  \end{array}$$
where $\norm . \norm $ is the usual Hilbert-Schmidt norm of linear operator. By the independence of the choice of the orthonormal basis we can express this norm in terms of the eigenvalues of the Ricci operator:
$$d[L,L]_{t} = \sum_{i} \lambda_{i}^{2}(T-t,X_{t}^{T}(x)) \,dt .$$
\end{proof}

\begin{remark}
We could choose $E_{i}$ such that $\tilde\nabla_{(0,e_{i})} (0,E_{i}(x))=0$ that do not change the martingale $L$, but give a simple version.
\end{remark}

\begin{remark}
 This martingale can be used to look at the behavior at point where the first singularity of the Ricci flow occurs, 
i.e. where the norm of the Riemannian curvature explodes (\cite{chowknopf},\cite{MTym}). 
\end{remark}

\bibliographystyle{plain}  
\bibliography{/home/kolehe/article_1/biblio}
\end{document}